\documentclass[reqno,final]{amsart}
\usepackage{amsaddr}
\usepackage{fancyhdr} 
\usepackage{color} 
\usepackage{hyperref} 
\usepackage{graphicx} 
\usepackage{amssymb}
\usepackage{pstricks}
\usepackage{amssymb}

\definecolor{aleacolor}{rgb}{0.16,0.59,0.78}

\hypersetup{
breaklinks,
colorlinks=true,
linkcolor=aleacolor,
urlcolor=aleacolor,
citecolor=aleacolor}

\theoremstyle{plain}

\newtheorem{lemma}{Lemma}

\theoremstyle{definition}

\theoremstyle{remark}

\usepackage{xspace}

\makeatletter \@addtoreset{equation}{section} \makeatother

 \usepackage{amsthm}
 
 \usepackage{amsmath}
 
 \usepackage{bbm}

 \usepackage{mathtools}

\newcounter{propcounter}
\usepackage{tikz}
\usepackage{tikz-3dplot}
\usetikzlibrary{calc,3d,arrows,shapes.geometric}
\usetikzlibrary{arrows.meta,calc,decorations.markings,math,arrows.meta}
\definecolor{mygreen}{RGB}{0,150,0}

\usepackage{pgfplots}
\pgfplotsset{compat=1.11}
\renewcommand{\paragraph}[1]{
     \textit{#1.} 
 }

\usepackage{array}
\usepackage{color}
\usepackage{multirow}
\usepackage{graphicx}
\usepackage{subcaption}
  \usepackage{scrextend}

\usepackage{algorithmicx}
\usepackage{algpseudocode}
\usepackage{algorithm}

\usepackage{bbm}
\usepackage{mathrsfs}
\usepackage{epsfig}
\usepackage{amssymb}
\usepackage{graphicx}
\usepackage{fancyhdr}
\usepackage{float}
\usepackage{epstopdf}

\usepackage{hyperref}

\hypersetup{
     linkcolor = blue,
     citecolor = blue,
     urlcolor = blue,
     anchorcolor = blue
    }

\usepackage{amsthm}

\usepackage{tikz}

\usepackage{pgfplots}
\pgfplotsset{compat=1.11}

\usetikzlibrary{arrows}

\usepackage{xspace}
\usepackage{float}
\usepackage{bbm}
\usepackage{amssymb}
\usepackage{amsmath}
\usepackage{amsthm}

\newcommand{\Prob}{{\mathbb{P}}\,}

\definecolor{lightgray}{gray}{0.9}

\definecolor{ltgray}{RGB}{200,200,200}
\definecolor{dkgray}{RGB}{150,150,150}

\newcommand{\todo}[1]{}

\begin{document}

\title[Some algebraic identity and its relations to Stirling numbers]{Some algebraic identity and its relations to Stirling numbers of the second kind}


\author{Paweł Lorek}

\address{Mathematical Institute\\ University of Wrocław}


\email{Pawel.Lorek@math.uni.wroc.pl}



 \begin{abstract}
 In this short note we provide some algebraic identity with a proof exploiting
 its  probabilistic interpretation.
 We show several consequences of the identity, in particular we obtain 
 a new representation of a Stirling number of second kind,
 $$
 S(n,d)={1\over d!} \sum_{1\leq j_1<j_2<\ldots<j_{d-1}< n}
 1\cdot2^{j_{d-1}-j_{d-2}}\cdots d^{j_1}$$
 for integers $n\geq d$. Relating this to other known formula for $S(n,d)$ we also obtain
 $$
 \sum_{1\leq j_1\leq j_2\leq \cdots\leq j_{n-d}\leq d} j_1j_2\ldots,j_{n-d}
=d! \sum_{1\leq j_1<j_2<\ldots<j_{d-1}< n}
 1\cdot2^{j_{d-1}-j_{d-2}}\cdots d^{j_1}.$$
 As a side effect, we have new proof of a known result stating that for any integer $d\in\mathbb{N}$
 and any $x\in\mathbb{R}$ equality 
 $$\sum_{r=0}^d (-1)^r{d\choose r}(x-r)^d=d!$$
 holds. This is a special case of the presented identity.
 \end{abstract}

\maketitle
\vspace{-0.3cm}
\section{Introduction and main result}
Stirling number of the second kind $S(n,d)$ is the number of ways to 
partition a set of $n$ objects into $d$ non-empty (thus $S(n,d)=0$ for $n<d$) subsets. 
An old classical result states that 
\begin{equation}\label{eq:Stir2}
 \sum_{r=0}^d(-1)^{r} {d\choose r} r^{n} =  (-1)^d d! S(n,d),
\end{equation}
see \textsl{e.g.},  \cite{Boyadzhiev} for an analysis point of view on $S(n,d)$, the above relation was already discussed in \cite{Gould}. Equation (\ref{eq:Stir2}) is often called the \textsl{Euler's formula}.
\smallskip\par 
A following formula 
\begin{equation}\label{eq:Ruiz}
 \sum_{r=0}^d (-1)^r{d\choose r}\left(x-r\right)^d = d!  
 \end{equation}
was first known for $x=0$ (see Eq. (1) in \cite{Boyadzhiev}).
 Ruiz \cite{Ruiz}  provided a  proof  by induction of above equality 
 for any $x\in\mathbb{R}$, later   Katsuura
 \cite{Katsuura} gave an elementary proof (also  for $x\in \mathbb{R}$).
 \smallskip\par\noindent
 If we define 
 $$g_{d,n}(x) := \sum_{r=0}^d (-1)^r{d\choose r}\left(x-r\right)^n \quad \textrm{and}\quad  
 \mathfrak{g}_{d,n}(x):=(-1)^{n\boldsymbol{1}(n>d)} g_{d,n}(x),$$
 then left hand sides of equations  (\ref{eq:Stir2}) and (\ref{eq:Ruiz}) can be 
 written as $\mathfrak{g}_{d,n}(0)$ and $\mathfrak{g}_{d,d}(x)$ respectively.
  \smallskip\par
 In this short note we provide an expression for $g_{d,n}(x)$ for any $n\geq d$ and $x\in\mathbb{R}$, thus, in a sense, provide an extension of equations  (\ref{eq:Stir2}) and (\ref{eq:Ruiz}). 
 As a consequence,   $i)$ we provide a new proof of equation (\ref{eq:Ruiz}); $ii)$ we provide a new representation of $S(n,d)$.

\noindent 
\begin{lemma}\label{lem:algeb_ident}
For fixed integers $n\geq d$ and for any $x\in \mathbb{R}$ we have (denoting $j_0\equiv 0$)

\begin{eqnarray}
f_{d,n}(x)&:=& \displaystyle x^{n-d} d! \sum_{1\leq j_1<j_2<\ldots<j_d\leq n} x^{d-j_d} \prod_{r=0}^{d-1} 
\left(x-(d-r) \right)^{j_{r+1}-j_{r}-1} \nonumber \\[12pt]
&=& \displaystyle 
\sum_{r=0}^d (-1)^r{d\choose r}\left(x-r\right)^n:=g_{d,n}(x). \label{eq:alg_ident}
\end{eqnarray}

\end{lemma}
\begin{proof}
First, we show that (\ref{eq:alg_ident}) holds for $x\geq d$. 
In such a case, we may divide both sides by $x^n$ and noting that 
$x^{d-j_d}=\prod_{r=0}^{d-1} (1/x)^{j_{r+1}-j_r}$  we may  rewrite it as


\begin{eqnarray}
\rho_{d,n}(x)&:=& \displaystyle {d!\over x^d} \sum_{1\leq j_1<j_2<\ldots<j_d\leq n} \prod_{r=0}^{d-1} \left(1-(d-r){1\over x}\right)^{j_{r+1}-j_{r}-1} \nonumber \\[6pt]
&=& \displaystyle {1\over x^n}
\sum_{r=0}^d (-1)^r{d\choose r}\left(x-r\right)^n:=\kappa_{d,n}(x). \label{eq:prob_ident}
\end{eqnarray}
 In such a case, it turns out that both sides have a probabilistic interpretation.
 We throw $n$ balls into  boxes $0,1,2,\ldots,d$, each ball 
 is independently placed with probability ${1\over x}$ in one of the boxes $1,2,\ldots,d$,
   with the remaining probability, \textsl{i.e.}, with probability $1-{d\over x}$, it is placed in 
 box  number 0. \par 
 We will compute the probability that each of the boxes $1,2,\ldots,d$ contains 
 at least one ball, denote this event by $B$, in two different ways.
 \smallskip\par 
 
\noindent $\bullet$ Method 1. Let $A_i$ be the event that box $i$ is empty. We have $\Prob(A_i)=\left(1-{1\over x}\right)^n$. The probability that boxes  ${i_1,\ldots,i_r}, r\leq d$ are empty is
  $$\Prob(A_{i_1}\cap\cdots\cap A_{i_r})=\left(1-{r\over x}\right)^n.$$
  From inclusion-exclusion formula, the probability that at least one box out of $1,\ldots,d$ is empty is
  \begin{eqnarray*}
\Prob\left(\bigcup_{r=1}^d A_i\right) & =& \sum_{\emptyset\neq J\subseteq \{1,\ldots,d\}}
  (-1)^{|J|+1}\Prob\left(\bigcap_{j\in J} A_j\right)   \\  
   & =& \sum_{r=1}^d (-1)^{r+1} {d\choose r} \left(1-{r\over x}\right)^n  \\
  \end{eqnarray*}
  and the probability that none of the boxes $1,\ldots,d$ is empty is
   \begin{eqnarray*}
\Prob\left(B\right) =1-\Prob\left(\bigcup_{r=1}^d A_i\right)
&=& \sum_{r=0}^d (-1)^{r} \left(1-{r\over x}\right)^n \\
& =&{1\over  x^n}\sum_{r=0}^d (-1)^{r} \left(x-r\right)^n=\kappa_{d,n}(x). 
  \end{eqnarray*}
  
  \medskip\par 
 
\noindent  $\bullet$  Method 2.  Assume that first box number $i_1$ becomes non-empty,
  then box number $i_2$ becomes non-empty, etc. until box $i_d$ becomes non-empty.
   Moreover, assume that box $i_1$ becomes non-empty at step $j_1$, box $i_2$ becomes 
   non-empty at step $j_2$, etc. until box $i_d$ becomes non-empty at step $j_d$.
   It means that for first $j_1-1$ steps all the boxes $1,\ldots,d$ were empty, \textsl{i.e.}, the balls were placed in box 0, what happens with probability $\left(1-{d\over x}\right)^{j_1-1}$.
   Then, at step $j_1$, a ball is placed in box $i_1$, what happens with probability $1/x$.
   Then, for next $j_2-j_1-1$ no new box (out of $1,\ldots,d$) becomes non-empty what happens 
   with probability $\left(1-{d-1\over x}\right)^{j_1-1}$, at step $j_2$ box $i_2$ 
   becomes non-empty with probability $1/x$ and so on. In general, box $i_r$ 
   becomes non-empty at step $j_r$ (with probability $1/x$), then no new box becomes non-empty 
   for $j_{r+1}-j_r-1$ steps, what happens with probability $\left(1-{d-(r-1)\over x}\right)^{j_2-j_1}$. 
   
The situation can be depicted as follows (upper rows -- step numbers, lower rows -- probabilities)

$$ \underbrace{1,\ldots,j_1-1}_{(1-d{1\over x})^{j_1-1}} \underbrace{j_1}_{{1\over x}}
 \underbrace{j_1+1,\ldots,j_2-1}_{(1-(d-1){1\over x})^{j_2-j_1-1}} \underbrace{j_2}_{{1\over x}} 
 \ldots 
  \underbrace{j_{r}+1,\ldots,j_{r+1}-1}_{(1-(d-r){1\over x})^{j_{r+1}-j_{r}-1}} \underbrace{j_{r+1}}_{{1\over x}} 
  \ldots $$
  $$\ldots
  \underbrace{j_{d-1}+1,\ldots,j_{d}-1}_{(1-{1\over x})^{j_d-j_{d-1}-1}} \underbrace{j_d}_{{1\over x}}   
  \underbrace{j_{d-1}+1,\ldots}_{1}
  $$
  Thus, the probability of the event is 
  \begin{equation}\label{eq:prod1}
    \prod_{r=0}^{d-1}\left({1\over x}\left(1-(d-r){1\over x}\right)^{j_{r+1}-j_r-1}\right)=
 {1\over  x^{d}}\prod_{r=0}^{d-1}\left(1-(d-r){1\over x}\right)^{j_{r+1}-j_r-1}.
  \end{equation}

  To compute the probability that none of the boxes $1,\ldots,d$ is empty,
  we need to sum (\ref{eq:prod1}) over all possible time steps $1\leq j_1<j_2<\cdots <j_d\leq n$ 
  at which consecutive boxes become non-empty and multiply by $d!$, since there are so many orderings of $i_1,\ldots,i_d$. Finally, we have 
  
    $$\Prob(B)={d!\over x^d}\displaystyle \sum_{1\leq j_1<j_2<\ldots<j_d\leq n} \prod_{r=0}^{d-1} \left(1-(d-r){1\over x}\right)^{j_{r+1}-j_{r}-1} $$
    and (\ref{eq:prob_ident}), and thus (\ref{eq:alg_ident}) for $x\geq d$ is proven.
  \medskip\par 
  Now note that 
 \begin{eqnarray}
f_{d,n}(x)&=&  \displaystyle  d! \sum_{1\leq j_1<j_2<\ldots<j_d\leq n} x^{n-j_d}\prod_{r=0}^{d-1}  \left(x-(d-r)\right)^{j_{r+1}-j_{r}-1}\label{eq:fdj_v2}
%
\end{eqnarray}
what means that $f_{d,n}(x)$ is (since $n-j_d\geq 0$) a polynomial. 
 To be more exact, it is a polynomial of degree $n-d$.
Thus, both $f_{d,n}(x)$ and $g_{d,n}(x)$ are polynomials of degree $n-d$. We have showed that
 $f_{d,n}(x)=g_{d,n}(x)$  for infinitely many points (for all $x\geq d$), what means that 
 they are equal for all $x\in\mathbb{R}$.
\end{proof}


 \section{Some identities arising from the main result.}
 Taking $n=d$ in Lemma \ref{lem:algeb_ident} we have (there is only one term in the some, since 
 we must have $j_1=1,\ldots,j_d=d=n$) that for any $x\in \mathbb{R}$
 \begin{equation*}\label{eq:ruiz}
 d!= \sum_{r=0}^d (-1)^r{d\choose r}\left(x-r\right)^d,
 \end{equation*}
 \textsl{i.e.}, formula (\ref{eq:Ruiz}) is recovered.
%
 Thus, the formula (\ref{eq:alg_ident}) can be seen as an extension of 
 (\ref{eq:Ruiz}), where we replace $(x-r)^d$ with    $(x-r)^n$ 
 for any  $n\geq d$.
 \medskip\par \noindent
 The following formula relating $g_{d,n}(x)$ and Stirling numbers 
 of the second kind
 \begin{equation}\label{eq:gdn2}
  g_{d,n}(x)=\sum_{r=0}^d (-1)^r{d\choose r}\left(x-r\right)^n 
 =  d! 
\sum_{k=d}^n{n\choose k} 
(-1)^{d-k}x^{n-k} S(k,d)
 \end{equation}
 is known --  \textsl{e.g.}, slightly different formulation 
is given in \cite[p. 254]{Boyadzhiev}. To derive it one needs to use 
binomial expansion of  $(x-r)^n$ and  formula  (\ref{eq:Stir2}).
Thus, we have a following representation
$$
\sum_{k=d}^n{n\choose k} 
(-1)^{d-k}x^{n-k} S(k,d)
=
\sum_{1\leq j_1<j_2<\ldots<j_d\leq n} x^{n-j_d}\prod_{r=0}^{d-1} \left(x-(d-r)\right)^{j_{r+1}-j_{r}-1}. $$

\medskip\par \noindent 
Relating (\ref{eq:Stir2}) and (\ref{eq:alg_ident}) 
we have  
\begin{eqnarray*}
  S(n,d) & = & {(-1)^{d}\over d!}\sum_{r=0}^d(-1)^r {d\choose r} r^{n} \\
  & = & {(-1)^{n-d}\over d!} \sum_{r=0}^d(-1)^r {d\choose r} (0-r)^{n} = g_{d,n}(0)=\left.{(-1)^d\over d!} f_{d,n}(x) \right|_{x=0}.  \\
\end{eqnarray*}
We thus need to compute the coefficient $a_0$ of 
a polynomial $f_{d,n}(x)=a_0+a_1 x + a_2 x^2+\ldots+a_{n-d}x^{n-d}$. Recall 
the formulation  (\ref{eq:fdj_v2})
$$f_{d,n}(x) = d! \sum_{1\leq j_1<j_2<\ldots<j_d\leq n} x^{n-j_d}\prod_{r=0}^{d-1}  \left(x-(d-r)\right)^{j_{r+1}-j_{r}-1}.$$
Only cases such that $j_d=n$ will contribute to $a_0$. Let us rewrite the product then 
$\prod_{r=0}^{d-1}  \left((r-d)+x)\right)^{j_{r+1}-j_{r}-1}$, the  intercept is (recall 
that $j_0=0$)
\begin{eqnarray*}
\prod_{r=0}^{d-1}  \left((r-d)+x)\right)^{j_{r+1}-j_{r}-1} &= &(0-d)^{j_1-1}\cdots (-2)^{j_{d-1}-j_{d-2}-1}\cdot(-1)^{n-j_{d-1}-1}\\
  &= & (-1)^{n-d} \cdot2^{j_{d-1}-j_{d-2}}\cdot 3^{j_{d-2}-j_{d-3}}\cdots d^{j_1} (d!)^{-1}
\end{eqnarray*}
and thus (now $j_d=n$)
$$  f_{d,n}(0)=(-1)^{n-d}\sum_{1\leq j_1<j_2<\ldots<j_{d-1}< n}
 1\cdot2^{j_{d-1}-j_{d-2}}\cdot 3^{j_{d-2}-j_{d-3}}\cdots d^{j_1} 
$$
Finally, we have a new representation of $S(n,d),$ namely
\begin{eqnarray}
 S(n,d) &= &{1\over d!} \sum_{1\leq j_1<j_2<\ldots<j_{d-1}< n}
 1\cdot2^{j_{d-1}-j_{d-2}}\cdots d^{j_1}. 
  \label{eq:newStirling2}
\end{eqnarray}
\smallskip\par 
Let us now define $S_2(n,d)$ to be the sum of the products of $n$ integers 
taken $d$ at a time with repetitions:
\begin{equation}\label{eq:S2}
S_2(n,d)=\sum_{1\leq j_1\leq j_2\leq \cdots\leq j_d\leq n} j_1j_2\ldots j_d.
\end{equation}
In \cite[Eq. (14.13) p. 195]{Gould2} the relation between $S_2(n,d)$ and 
Stirling numbers of second kind $S(n,d)$ was provided:
\begin{equation}\label{eq:S2S}
S(n,d)=S_2(d,n-d).
\end{equation}
For example, we have $S(6,4)=S_2(4,2)$, computing it using (\ref{eq:newStirling2}) 
and (\ref{eq:S2}) yields

$$\begin{array}{llllllll}
 S(6,4) & = & \displaystyle{1\over 4!} & \large( & 
 2\cdot3\cdot 4 + 2^2\cdot3\cdot 4  + 2^3\cdot3\cdot 4  + 2\cdot3^2\cdot 4  + 2^2\cdot3^2\cdot 4 \ + & \\[8pt]
        &    &            &    & 
  2\cdot3^3\cdot 4  + 2\cdot3\cdot 4^2  + 2^2\cdot3\cdot 4^2  +2\cdot3^2\cdot 4^2  + 2\cdot3\cdot 4^3
  \left.\right)\\[8pt]
  & = & \displaystyle{1\over 24} & \large( & 
  24+48+96+72+144+216+96+192+288 \quad \large)   = \displaystyle {1560\over 24}=65.
    & \\[17pt]   
    
     S_2(4,2) & = &  &   & 
     1\cdot 1 + 1\cdot 2 + 1\cdot 3 + 1\cdot 4 + 2\cdot 2 + 2\cdot 3 + 2\cdot 4 + 3\cdot 3\ + 
     \\[8pt]
        &    &            &    & 
         3\cdot 4 + 4\cdot 4 = 1+2+3+4+4+6+8+9+12+16 = 65.
\end{array}
$$
\medskip\par 
\noindent
Combining (\ref{eq:newStirling2}), (\ref{eq:S2}) and  (\ref{eq:S2S}) we get 
for $n\geq d$
$$
\sum_{1\leq j_1\leq j_2\leq \cdots\leq j_{n-d}\leq d} j_1j_2\ldots,j_{n-d}
=d! \sum_{1\leq j_1<j_2<\ldots<j_{d-1}< n}
 1\cdot2^{j_{d-1}-j_{d-2}}\cdots d^{j_1},$$
 or more compactly (where $j_d$ can be any number on the rhs),
$$
\sum_{1\leq j_1\leq j_2\leq \cdots\leq j_{n-d}\leq d} \prod_{r=1}^{n-d} j_r
=d! \sum_{1\leq j_1<j_2<\ldots<j_{d-1}< n}
\prod_{r=1}^d r^{j_{d-r+1}-j_{d-r}}.$$
\smallskip\par \noindent
Another expression for $S_2(n,d)$ is provided in \cite[Eq. (14.10) p. 194]{Gould2}, namely 
$$S_2(n,d)=\sum_{j_d=1}^n j_d  \sum_{j_{d-1}=j_d}^n j_{d-1} \sum_{j_{d-2}=j_{d-1}}^n j_{d-2} \cdots j_2 \sum_{j_1=j_2}^n j_1 .$$
Thus, we also have a relation 
\begin{eqnarray*}
 & & \sum_{j_{n-d}=1}^d j_{n-d}  \sum_{j_{n-d-1}=j_{n-d}}^d j_{n-d-1} \sum_{j_{n-d-2}=j_{n-d-1}}^n j_{n-d-2} \cdots j_2 \sum_{j_1=j_2}^d j_1 \\[12pt]
 &=&d! \sum_{1\leq j_1<j_2<\ldots<j_{d-1}< n}
 1\cdot2^{j_{d-1}-j_{d-2}}\cdots d^{j_1}.
\end{eqnarray*}

\smallskip\par 

It is worth noting, that Batir \cite{Batir} showed a following 
relation of  $S(n,d)$ with $s_n(d)$, the latter involves \textsl{similar} sums
to the ones appearing in $S_2(n,d)$ (however, the latter 
sums the products of reciprocals of $j_1,\ldots,j_d$), namely
\begin{eqnarray*}
S(n,d)&=&{-1\over d!} s_n(-d), \qquad  \textrm{where}  \\
s_n(d) &=& \sum_{r=1}^d{d\choose r} {(-1)^{r-1}\over r^n}\stackrel{(*)}{=} \sum_{1\leq j_1\leq j_2\leq\ldots\leq j_n\leq d} {1\over j_1 j_2\cdots j_d},
\end{eqnarray*}
where $(*)$ was proven in \cite{Dilcher}.
Note the difference: on the rhs the roles of $n$ and $d$ are swapped, the summation is over 
indices with weak inequalities.
As a side effect, we also obtain
\begin{equation*}
 \sum_{1\leq j_1<j_2<\ldots<j_{d-1}< n}
 1\cdot2^{j_{d-1}-j_{d-2}}\cdots d^{j_1-1} 
 =(-1)^{n-d-1} s_n(-d).
\end{equation*}

\medskip\par \noindent
Using (\ref{eq:gdn2}) and (\ref{eq:alg_ident}) we may rewrite 

 \begin{equation*} 
  d! \sum_{k=d}^n{n\choose k} 
(-1)^{d-k}x^{n-k} S(k,d) = f_{d,n}(x),
 \end{equation*}
 equivalently, for $x\neq 0$
  \begin{equation*} 
   \sum_{k=d}^n{n\choose k} 
(-1)^{n-k}{S(k,d)\over x^k} = (-1)^{n-d} {f_{d,n}(x)\over d! x^n}=:\beta_{n,d}(x).
 \end{equation*}
 Recall the famous \textsl{inversion formula}:
 $$a_n=\sum_{k=0}^n {n\choose k}(-1)^{n-k}b_k \qquad \textrm{iff} \qquad 
 b_n=\sum_{k=0}^n {n\choose k} a_k.$$
 Taking $b_k={S(k,d)\over x^k}$ and $a_n=\beta_{n,d}(x)$ we obtain  
 $${S(n,d)\over x^n}=\sum_{k=0}^n {n\choose k} (-1)^{n-k}{f_{k,n}(x)\over k! x^n},$$
 \textsl{i.e.}, we obtain another formula for $S(n,d)$ in terms of functions $f_{k,n}(x)$, namely:
 $$S(n,d)=\sum_{k=0}^n {n\choose k} (-1)^{n-k}{f_{k,n}(x)\over k!}.$$
 
\bigskip\par 
 
\bibliographystyle{amsalpha} 
\bibliography{Lorek_algebraic_identity}

 \end{document}